\documentclass{amsart}
\usepackage{amsmath,amssymb}
\usepackage{graphicx,tikz}
\newtheorem{theorem}{Theorem}[section]

\newtheorem{lemma}[theorem]{Lemma}

\newcommand\restr[2]{{
		\left.\kern-\nulldelimiterspace 
		#1 
		\vphantom{\big|} 
		\right|_{#2} 
}}

\theoremstyle{definition}

\theoremstyle{remark}
\newtheorem{remark}[theorem]{Remark}

\newcommand{\al}{\alpha}
\newcommand{\be}{\beta}
\newcommand{\de}{\delta}
\newcommand{\ep}{\varepsilon}

\newcommand{\ga}{\gamma}
\newcommand{\ka}{\kappa}
\newcommand{\la}{\lambda}
\newcommand{\om}{\omega}
\newcommand{\si}{\sigma}

\newcommand{\De}{\Delta}

\newcommand{\Om}{\Omega}

\def\RR{\mathbb{R}}

\def\TT{\mathbb{T}}
\def\HH{\mathbb{H}}

\renewcommand\SS{\mathbb{S}}
\newcommand{\cO}{{\mathcal O}}

\newcommand{\pd}{\partial}
\newcommand\minus\backslash

\newcommand{\ms}{\mspace{1mu}}
\newcommand\lan\langle
\newcommand\ran\rangle

\newcommand{\supp}{\operatorname{supp}}

\DeclareMathOperator\dist{dist}

\renewcommand\leq\leqslant
\renewcommand\geq\geqslant
\newlength{\intwidth}

\newcommand\coul{^{\mathrm{C}}}
\newcommand\any{^{\mathrm{A}}}
\newcommand\harm{^{\mathrm{H}}}

\newcommand\BOm{\overline\Om}

\DeclareMathOperator{\spn}{span}

\newcommand\dVol{ d\ms\mathrm{Vol}}

\begin{document}

\title[Schr\"odinger operators on hyperbolic spaces]{Inverse localization and global approximation\\ for some Schr\"odinger operators\\ on hyperbolic spaces}

 \author{Alberto Enciso}
 \address{Instituto de Ciencias Matem\'aticas, Consejo Superior de
   Investigaciones Cient\'\i ficas, 28049 Madrid, Spain}
 \email{aenciso@icmat.es}

 \author{Alba Garc\'\i a-Ruiz}
\address{Instituto de Ciencias Matem\'aticas, Consejo Superior de
   Investigaciones Cient\'\i ficas, 28049 Madrid, Spain}
 \email{alba.garcia@icmat.es}

 \author{Daniel Peralta-Salas}
 \address{Instituto de Ciencias Matem\'aticas, Consejo Superior de
   Investigaciones Cient\'\i ficas, 28049 Madrid, Spain}
 \email{dperalta@icmat.es}

%
%
\begin{abstract}
We consider the question of whether the high-energy eigenfunctions of certain Schr\"odinger operators on the $d$-dimensional hyperbolic space of constant curvature $-\ka^2$ are flexible enough to approximate an arbitrary solution of the Helmholtz equation $\Delta h+h=0$ on $\RR^d$, over the natural length scale $\cO(\la^{-1/2})$ determined by the eigenvalue $\la\gg1$. This problem is motivated by the fact that, by the asymptotics of the local Weyl law, approximate Laplace eigenfunctions do have this approximation property on any compact Riemannian manifold. In this paper we are specifically interested in the Coulomb and harmonic oscillator operators on the hyperbolic spaces $\HH^d(\ka)$. As the dimension of the space of bound states of these operators tends to infinity as $\ka\searrow0$, one can hope to approximate solutions to the Helmholtz equation by eigenfunctions for some $\ka>0$ that is not fixed a priori. Our main result shows that this is indeed the case, under suitable hypotheses. We also prove a global approximation theorem with decay for the Helmholtz equation on manifolds that are isometric to the hyperbolic space outside a compact set, and consider an application to the study of the heat equation on~$\HH^d(\ka)$. Although global approximation and inverse approximation results are heuristically related in that both theorems explore flexibility properties of solutions to elliptic equations on hyperbolic spaces, we will see that the underlying ideas behind these theorems are very different.
\end{abstract}

\maketitle

\section{Introduction}

In this paper we are concerned with high energy eigenfunctions of Schr\"odinger operators on a complete Riemannian $d$-dimensional manifold~$M$ with $d\geq 2$, which we will eventually choose to be a hyperbolic space. To this end, we consider nontrivial square-integrable solutions $\psi_j$ to equations of the form
\begin{equation}\label{Schr}
\De_M \psi_j+(\la_j-V)\psi_j=0\,,
\end{equation}
where $\De_M$ is the Laplace--Beltrami operator on the manifold, the potential~$V$ is a well-behaved function and $\la_j\gg1$ are eigenvalues. It is well known that rescaled limits of high energy eigenfunctions of a Schr\"odinger operator are connected with solutions to the Helmholtz equation. Recall that the Helmholtz equation reads as
\begin{equation*}
	\De_M u+\la u=0
\end{equation*}
for some constant $\la>0$, and that in the $d$-dimensional Euclidean space one can often take advantage of the scaling properties of the equation and set $\la=1$:
\begin{equation}\label{Helm}
\De v+ v=0\,.	
\end{equation}

It is standard that, over geodesic balls of the natural radius determined by the eigenvalue, high energy eigenfunctions on a manifold, such as solutions to~\eqref{Schr}, behave essentially like solutions to the Euclidean Helmholtz equation~\eqref{Helm} on the unit ball. Roughly speaking, this is because, if $\bar x$ are normal coordinates on a geodesic ball centered at a certain point of the manifold and one considers the rescaled coordinates $x:=\la_j^{1/2}\bar x$, a straightforward computation shows that Equation~\eqref{Schr} can be rewritten on the Euclidean ball $|x|<1$ as
\begin{equation*}
\De \psi_j + \psi_j = \mathcal{O}\left(\la_j^{-1}\right)\,,
\end{equation*}
where $\De$ denotes the ordinary Laplacian in the coordinates~$x$ and where the error term depends on derivatives of~$\psi_j$ up to second order.

It is also known that, under mild assumptions, one can pick a sequence of approximate eigenfunctions on the manifold (meaning, for concreteness, a linear combination~$U_\la$ of eigenfunctions with eigenvalues $\la_j$ lying on the interval $I_\la:=[\la,(1+\de)\la]$, with $\la\gg1$ and some arbitrarily small but fixed $\de>0$) whose behavior on a ball of radius $\la^{-1/2}$ centered at a fixed point~$p$ reproduces that of any fixed solution~$v$ to the Helmholtz equation~\eqref{Helm} on~$\RR^d$ modulo a small error. More precisely, by the well-known asymptotics for the local Weyl law~\cite{Hormander}, given any solution~$v$, one can construct $U_\la$ as above such that
\begin{equation}\label{Ula}
\left\|U_\la\circ \exp^M_p(\la^{-1/2}\,\cdot)-v\right\|_{C^k(B)}\to 0
\end{equation}
as $\la\to\infty$, where $\exp_p^M$ is the exponential map at a point~$p\in M$, $B\subset\RR^d$ is a bounded domain and $k$ is any fixed integer. (In fact, it is well known that the size of the interval $I_\la$ can be made substantially smaller, see e.g.~\cite{CanHanin}.) Nontrivial extensions of this property, which consider random combinations of eigenfunctions with eigenvalues in the interval $I_\la$, can be consulted in~\cite{CanSar} and in the references therein.

A subtler question is whether one can replace the approximate eigenfunctions~$U_\la$ by bona fide eigenfunctions~$\psi_j$. Remarkably, there are only a few cases where it is known that high-energy eigenfunctions can approximate, on suitable scales, arbitrary solutions to the Helmholtz equation.  For short, when this happens, we will say that the Schr\"odinger operator $-\De_M+V$ has the {\em inverse localization property}\/. In general, one does not expect that this should hold for a general Riemannian manifold, even when the potential is identically~0.

In fact, Jung and Zelditch have constructed~\cite{Zelditch} a Riemannian $3$-manifold for which all the (nonconstant) Laplace eigenfunctions have exactly two nodal domains, and this can be used to show that there are solutions to the Helmholtz equation that cannot be approximated by rescaled eigenfunctions on this manifold in the sense of~\eqref{Ula}. Also, in \cite{Yo} we have recently shown that the inverse localization property does not hold on generic flat tori. Specifically, whether the approximating property holds or not depends solely on the arithmetic properties of the spectrum of the torus, and the set of tori for which it holds has measure zero.

The list of Schr\"odinger operators which are known to possess the inverse localization property is remarkably short. To our best knowledge, the only known examples are the harmonic oscillator operator on~$\RR^d$~\cite{RMI} and the Laplace operator on certain tori~\cite{Yo} and on the round sphere $\SS^d$ and all Riemannian quotients thereof~\cite{refs}. The Coulomb operator~\cite{JEMS} exhibits an analogous localization property, where the Helmholtz equation must be replaced by a ``zero energy Coulomb equation'' to account for the fact that one considers eigenvalues that do not tend to infinity but to~$0$, which is the bottom of the essential spectrum. The curl operator on~$\TT^3$ and~$\SS^3$, whose eigenfunctions are known as Beltrami fields, also exhibits an inverse localization property~\cite{ASENS}. These facts have found several applications in previous work of two of the authors, including the proof of a conjecture of Berry~\cite{Berry} on eigenfunctions whose nodal sets realize a certain knot~\cite{JEMS,RMI} or the construction of stationary Euler flows with knotted vortex lines~\cite{ASENS}.

Note that in all the known examples, the eigenvalues have a very high multiplicity. As shown by Uhlenbeck~\cite{Uhlenbeck}, a generic perturbation of these examples will split the eigenspaces, leading to Schr\"odinger operators for which all the eigenvalues have multiplicity~1 and for which, in general, one does not expect inverse localization to hold. However, it stands to reason that some kind of inverse localization property should hold provided that the perturbation preserves the multiplicity of the eigenvalues.

Our objective in this paper is to explore this heuristic idea in the context of  Schr\"odinger operators on the $d$-dimensional hyperbolic space $\HH^d(\ka)$ of negative constant sectional curvature $-\ka^2$. Specifically, we are interested in the Coulomb and harmonic oscillator operators
\begin{equation*}
H_\ka\coul :=-\De_\ka+ \al V_\ka\coul\,,\qquad H_\ka\harm :=-\De_\ka+ \al V_\ka\harm\,,
\end{equation*}
where $\De_\ka$ is the Laplace--Beltrami operator, $V_\ka\coul$ and $V_\ka\harm$ are the Coulomb and harmonic oscillator potentials on $\HH^d(\ka)$ and $\al$ is a fixed positive constant that will not play any role in the arguments. Precise definitions and explicit formulas will be provided in Section~\ref{S.eigenfunctions}.

These operators have long played a significant role in mathematics and in physics. Indeed, the study of the classical counterparts of these systems can be traced back~\cite{Shc} to Lobachevski circa 1835 in the case of the Coulomb (or Kepler) potential and to Liebman in 1902 in the case of the harmonic oscillator. The study of the corresponding quantum models respectively started in Schr\"odinger in 1940 and with Higgs in 1979. It has long been known that these models possess superintegrability properties analogous to those of their Euclidean analogs, as discussed in a broader context in~\cite{Ball}. One should note that, even though the Coulomb system is typically considered in any dimension $d\geq2$, the potential is always a formal extension of the 3-dimensional case in the sense that it diverges as the inverse of the distance to the point $p_0$ where the point charge is located. In particular, the equation $\De_\ka V_\ka\coul= c_{d,\ka}\, \de_{p_0}$, where $\de_{p_0}$ is the Dirac measure supported on~$p_0$, only holds in dimension~3.

Even though the Coulomb and harmonic oscillator operators on $\HH^d(\ka)$ are formally very similar to those of their Euclidean counterparts, their spectral properties are rather different. In particular, the space of bound states is finite dimensional for both operators. We will elaborate on this point in Section~\ref{S.eigenfunctions}. An easy consequence of this fact is that the operators $H_\ka\any$, with $\mathrm{A}=\mathrm{C}$ or $\mathrm{H}$, cannot have the inverse localization property for any $\ka>0$.

However, the main result of this paper is that a suitable analog of the inverse localization principle for parametric families of operators does indeed hold. The basic idea is that, as the dimension of the space of bound states of $H_\ka\any$ tends to infinity as $\ka\searrow 0$, one can try to approximate solutions to the Helmholtz equation by eigenfunctions of $H_\ka\any$ for some~$\ka>0$ that is not fixed a priori. For simplicity, in the Introduction we only deal with the harmonic oscillator operator.  As in the Euclidean case, the analogous statement for the Coulomb operator $H_\ka\coul$ is necessarily less transparent, so we have chosen relegate it to Theorem~\ref{T.Coul} in the main text. To state the theorem, let us fix a point~$p_0\in \HH^d(\ka)$ and denote by $\exp_{p_0}^\ka$ the exponential map at~$p_0$ defined by the hyperbolic metric of curvature $-\ka^2$.

\begin{theorem}\label{T1}
	Let $v$ be an even or odd solution to the Helmholtz equation~\eqref{Helm} on $\RR^d$, with $d\geq2$. Given any $\ep>0$ small, any integer $k$, and any ball $B\subset\mathbb{R}^d$, there exist some $\ka>0$ and an eigenfunction $\psi$ of $H_\ka\harm$ such that
	\begin{equation*}
	\|\psi\circ \exp_{p_0}^\ka(\lambda^{-1/2} \,\cdot)-v\|_{C^k(B)}<\ep\,.
	\end{equation*}
	Here $\lambda$ is the corresponding eigenvalue and $p_0\in \HH^d(\ka)$ is the point where the potential $V_\ka\harm$ attains its global minimum.
\end{theorem}

%

Let us point out that assuming that the function~$v$ in Theorem~\ref{T1} satisfies the Helmholtz equation in all of~$\RR^d$ or in a smaller domain is actually irrelevant, as it is whether one aims to approximate the function on a ball or on any other bounded domain~$\Om$ whose complement $\RR^d\backslash \overline{\Om}$ is connected. This is because one can prove~\cite{Acta15} that a function satisfying the Helmholtz equation on a domain containing the closure of a domain~$\Om$ as above, which would be an essentially minimal assumption for the purposes of Theorem~\ref{T1}, can be approximated in the $C^k(\Om)$ norm by a solution of the Helmholtz equation on the whole~$\RR^d$ with the sharp fall off rate at infinity (i.e., $|v(x)|<C(1+|x|)^{(1-d)/2}$ if $d\geq2$).

Let us provide some context for this result. It is a general fact that a solution~$w$ of a nice linear elliptic equation $Lw=0$ on a domain containing the closure of a bounded domain~$\Om$ with connected complement (be it on $\RR^d$ or on a noncompact manifold) can be approximated in $C^k(\Om)$ by a global solution, that is, a solution to the equation on the whole space. This kind of results, which generalize Runge's classical theorem in complex analysis, goes back to the work of Lax, Malgrange and Browder, and are known as {\em global approximation theorems}\/~\cite{Browder}. The key property of the Helmholtz equation on~$\RR^d$, established in~\cite{Acta15}, is that one can additionally assume a decay condition at infinity for the global solution, which is moreover sharp.

As we shall see next, one can also prove a global approximation theorem with decay for the Helmholtz equation on a hyperbolic space. However, it is worth stressing that the underlying ideas behind this result and behind an inverse localization theorem are rather different. To illustrate this fact, we shall state the following more general result, which does not follow from the approximation theorems with decay for essentially flat elliptic operators proved in~\cite{Duke}. Recall that a $d$-dimensional Riemannian manifold $M$ is {\em hyperbolic outside a compact set}\/ if there exist compact sets $K_1\subset M$ and $K_2\subset\HH^d(\ka)$ such that $M\backslash K_1$ and $\HH^d(\ka)\backslash K_2$ are isometric, for some $\ka>0$. We will consider manifolds for which there is a point $p_0\in M$ whose injectivity radius is infinite, which implies in particular that $M$ is diffeomorphic to $\RR^d$. An important class of manifolds with this property are those of Cartan--Hadamard, that is, simply connected manifolds with  nonpositive sectional curvature. To state this result, we will denote by $\rho_M(x)=\rho_M(x,p_0)$ the geodesic distance in~$M$ from $x$ to the point $p_0\in M$.

\begin{theorem}\label{T2}
Consider a $d$-dimensional Riemannian manifold $M$ with $d\geq 2$ which is hyperbolic outside a compact set and has a point $p_0$ with infinite injectivity radius. Fix an integer~$k$ and a constant~$\ep>0$. Suppose that the function~$w$ satisfies the Helmholtz equation
	\begin{equation}\label{HelmM}
	\De_M w+ \la w=0
	\end{equation}
	in a neighbourhood of the closure of a bounded domain~$\Om\subset M$
	 , and that the complement $M\backslash \BOm$ is connected. We also assume that $\lambda$ belongs to the continuous spectrum of the Laplacian, i.e., $\la>\left[\frac{(d-1)\ka}{2}\right]^2$. Then there exists a solution~$v$ of the Helmholtz equation~\eqref{HelmM} on all of~$M$ which approximates~$w$ as
	\begin{equation*}
	\|w-v\|_{C^k(\Om)}<\ep
	\end{equation*}
	and which satisfies the sharp Agmon--H\"ormander decay condition
\begin{equation} \label{eq.decayAH}
\sup _{R>0} \frac{1}{R} \int_{B_{R}}|v(x)|^{2} \,\dVol(x)<\infty \,.
\end{equation}
Here $B_R$ is the geodesic ball of radius $R$ centered at $p_0$ and $\dVol$ is the Riemannian volume form on $M$. In the case that $M$ is the hyperbolic space $\HH^d(\kappa)$, the decay of $v$ is pointwise:
\begin{equation}\label{eq.decaypoint}
	|v(x)|< Ce^{-\frac{d-1}{2}\ka \rho_M(x)}.
\end{equation}
\end{theorem}

\begin{remark}
Using that $M$ is isometric to the hyperbolic space outside some compact set, and the expression of $\dVol$ in geodesic coordinates centered at $p_0$, it is straightforward to check that the decay~\eqref{eq.decayAH} is an averaged version of the pointwise decay~\eqref{eq.decaypoint}. We stress that the condition $\lambda>\left[\frac{(d-1)\ka}{2}\right]^2$, where $\left[\frac{(d-1)\ka}{2}\right]^2$ is the bottom of the purely continuous spectrum of $\Delta_M$, is only used to ensure the decay bound~\eqref{eq.decayAH}; a global approximation theorem with no control at infinity of the global solution certainly holds without the aforementioned assumption.
\end{remark}

\begin{remark}
Note that we are not assuming that the curvature of~$M$ is asymptotically small. Furthermore, a result of Uhlenbeck~\cite{Uhlenbeck} ensures that all the eigenvalues of the Dirichlet Laplacian are nondegenerate for generic compact perturbations of~$\HH^d(\ka)$. Therefore, the validity of this result (unlike that of Theorem~\ref{T1}) is in no way related to the multiplicity of the eigenvalues.
\end{remark}

The decay condition of a global approximation theorem is key in many applications, such as nonlinear equations~\cite{Allen}. Although we shall not pursue this direction here, in Section~\ref{S.GAT} we will exploit this fact to derive a global approximation theorem with decay for the heat equation on $\HH^d(\ka)$ (Theorem~\ref{T.heat}).

\section{The harmonic oscillator and Coulomb operators in hyperbolic spaces}
\label{S.eigenfunctions}
We will describe the $d$-dimensional hyperbolic space $\HH^d(\ka)=(\RR^d,g_\ka)$ of sectional curvature $-\ka^2<0$ via geodesic normal coordinates centered at a fixed but arbitrary point $p_0\in\HH^d(\ka)$. That is, we take coordinates $\rho=\rho_M(x,p_0):=\dist_\ka(x,p_0)$, the geodesic distance to $p_0$, and the angular variable $\omega\in\SS^{d-1}$. In these coordinates the metric takes the form $g_\ka=d\rho^2+\frac{\sinh^2(\ka\rho)}{\ka^2}g_{\SS^{d-1}}$. For simplicity, in Sections \ref{S.eigenfunctions}--\ref{S.Coul} we will instead use the coordinates $(r,\om)$ with \begin{equation*}r:=\frac{\sinh(\ka\rho)}{\ka}\,.\end{equation*}
In these coordinates, the metric and the Laplacian read as
\begin{align*}
g_\ka&=\frac{1}{1+\ka^2r^2}dr^2+r^2g_{\SS^{d-1}}\,,\\
\De_\ka &=(1+\ka^2r^2)\frac{\partial^2 }{\partial r^2}+\frac{d-1+d\ka^2r^2}{r}\frac{\partial }{\partial r}+r^{-2}\De_{\SS^{d-1}}\,.
\end{align*}

We are interested in the Coulomb and harmonic oscillator potentials on~$\HH^d(\ka)$, which are given respectively by
\begin{align*}
	V_\ka\coul&:= -\sqrt{\ka^2+r^{-2}}\,,\\
	 V_\ka\harm&:=\frac{r^2}{1+\ka^2r^2}\,.
\end{align*}
The point $p_0$ is chosen to be the global minimum point of the harmonic potential and the singular point of the Coulomb potential.
Note that when $d=3$, $\Delta_\ka V_\ka\coul=c_{d,\ka} \delta_{p_0}$, where $c_{d,\ka}$ is an explicit constant, so $V_\ka\coul$ on $\HH^3(\ka)$ has a direct interpretation as the electrostatic potential of a point charge located at $p_0$. Also, the harmonic potential is defined as $V_\ka\harm:=\left(V_\ka\coul\right)^{-2}$, which mimics the well known functional relation between the harmonic potential $|x|^2$ and the Coulomb potential $|x|^{-1}$ in Euclidean space.

The eigenvalues and eigenfunctions of the Coulomb operator $H_\ka\coul$ are well known (see e.g.~\cite{Ques}). An orthogonal basis of the space of square-integrable eigenfunctions is
\begin{equation*}
	\begin{aligned}
		\psi^{\kappa,\mathrm{C}}_{nlm}&:= f^{\kappa,\mathrm{C}}_{nl}(r)Y_{lm}(\omega)\,,\\
		f^{\kappa,\mathrm{C}}_{nl}(r)&=\ r^{l}\left(\sqrt{1+\ka^2r^2}+\ka r\right)^{c} P_{n}^{(a;b)}\left(2\left(\sqrt{1+\ka^2r^2}+\ka r\right)^2-1\right)\,.		 
	\end{aligned}
\end{equation*}
The constants are defined in terms of $n$ and~$l$ as
\begin{equation*}
		c:=-n-\frac{\al}{2\ka(n+l+\frac{d-1}{2})};\ a:=2l+d-2;\ b:=-n-l-\frac{d-1}{2}-\frac{\al}{2\ka(n+l+\frac{d-1}{2})}\,,
		\end{equation*}	
and hereafter $Y_{lm}$, with $1\leq m\leq d_l:={l+d-2\choose l}\frac{2l+d-2}{l+d-2}$, denotes any orthonormal basis of the space $(d-1)$-dimensional spherical harmonics with spherical eigenvalue $\mu_l=l(l+d-2)$, for any nonnegative integer~$l$. Also, $P_{n}^{(a;b)}$ is a Jacobi polynomial. The indices $n,l$ range over the set of nonnegative indices such that
\begin{equation}\label{rCE}
n+l<\sqrt{\frac{\al}{2\ka}}-\frac{d-1}{2},
\end{equation}and $1\leq m\leq d_l$. In particular, the point spectrum is finite, contrary to what happens in the Euclidean case of the Euclidean Coulomb operator.

The eigenvalue of $\psi^{\kappa,\mathrm{C}}_{nlm}$ is
\begin{equation*}\lambda_{nl}^{\kappa,\mathrm{C}}=-\frac{\al^2}{4(n+l+\frac{d-1}{2})^2}-\ka^2\left(n+l+\frac{d-1}{2}\right)^2\,.\end{equation*}
Just as in the Euclidean counterpart, the energy only depends on $N:=n+l$. Thus the multiplicity of $\lambda_{nl}$ is $\sum_{j=0}^{N}d_j$, just as in the Euclidean case.

In the case of the harmonic oscillator $H_\ka\harm$, a basis of the space of bound states is
\begin{equation*}
	\begin{aligned}
		\psi^{\ka,\mathrm{H}}_{nlm}&:= f^{\ka,\mathrm{H}}_{nl}(r)Y_{lm}(\omega)\,,\\
		f^{\ka,\mathrm{H}}_{nl}(r)&:=\ r^{l}(1+\ka^2r^2)^{\frac{c}{2}} P_{n}^{(a;b)}\left(1+2\ka^2r^2\right)\,,
		\end{aligned}
		\end{equation*}
with
\begin{equation*}
c:=-\be/\ka^2;\ a:=l+\frac{d-2}{2};\ b:=-\be/\ka^2-1/2\,.
\end{equation*}
Here the nonnegative integers range over the set
\begin{equation}\label{rHE}
		2n+l<\frac{\be}{\ka^2}-\frac{d-1}{2},\end{equation}and $1\leq m\leq d_l$. Thus the space of bound states is again finite-dimensional.
The eigenvalue of		$\psi^{\ka,\mathrm{H}}_{nlm}$ is
		\begin{equation*}
\lambda_{nl}^{\ka,\mathrm{H}}=\be(4n+2l+d)-\ka^2\left(2n+l+\frac{d-1}{2}\right)^2\,,
\end{equation*}
where $\be:=\frac{\sqrt{\ka^4+4\alpha}-\ka^2}{2}$. This depends only on $N:=2n+l$, so the multiplicity of $\lambda_{nl}^{\ka,\mathrm{H}}$ is then as in the Euclidean case as well:
\begin{equation*}
\sum_{k=0}^{\left\lfloor\frac{N}{2}\right\rfloor} d_{N-2 k}\,.
\end{equation*}

It is worth recalling that, in Euclidean space, an orthogonal basis of eigenfunctions of the space of bound states of the Coulomb and harmonic oscillator operators,
\begin{align*}
H^{\mathrm{H}}&:=-\De +\al|x|^2\,,\\
H^{\mathrm{C}}&:=-\De -\al|x|^{-1}\,,\\
\end{align*}
is respectively given by
\begin{align*}
	\psi\harm_{nlm}&:=r^l e^{-\sqrt{\alpha}r^2/2}L_n^{(l+\frac{d}{2}-1)}\left(\sqrt{\alpha}r^2\right)Y_{lm}(\omega)\,,\\
	\psi\coul_{nlm}&:=r^l e^{-\al r/2\left(n+l+\frac{d-1}{2}\right)}L_n^{(2l+d-2)}\left(\frac{\alpha r}{n+l+\frac{d-1}{2}}\right)Y_{lm}(\omega)\,.
\end{align*}
Here $n,l\geq0$ and $1\leq m\leq d_l$. The corresponding eigenvalues are
\begin{align*}
	\lambda_{nl}\harm=\sqrt{\al}(4n+2l+d)\,,\qquad \lambda_{nl}\coul:=-\frac{\al^2}{4\left(n+l+\frac{d-1}{2}\right)^2}\,.
\end{align*}

\section{Proof of Theorem~\ref{T1}}
\label{S.T1}
Assume that $v$ is an even function, i.e., $v(x)=v(-x)$ for all $x\in\RR^d$. The case when $v$ is odd  is  analogous and will be sketched at the end of this section.

Let us begin by  taking hyperspherical coordinates in a ball $B'\supset\!\supset B$. For any fixed~$r$, one can expand~$v$ in hyperspherical harmonics to write
\begin{equation*}
v= \sum_{l\geq0}\sum_{m=1}^{d_l} v_{lm}(r)Y_{lm}(\om)\,,
\end{equation*}
with
\begin{equation*}
v_{lm}(r)=\int_{\SS^{d-1}}v(r,\om)\, Y_{lm}(\om)\, d\si(\om)\,.
\end{equation*}
Furthermore, since
\begin{equation*}
\De_{\SS^{d-1}}Y_{lm}=-\mu_l Y_{lm}= -l(l+d-2) Y_{lm}\,,
\end{equation*}
one can integrate the Helmholtz equation~\eqref{Helm}
\begin{equation*}
- v
= \De v= \pd_{rr}v +\frac{d-1}r\pd_r v +\frac{\De_{\SS^{d-1}}v}{r^2}\end{equation*}
with $Y_{lm}$ over $\SS^{d-1}$ to obtain
\begin{align*}
- v_{lm}	&= \int_{\SS^{d-1}} Y_{lm}\,\De v= \left(\pd_{rr} +\frac{d-1}r\pd_r \right)\int_{\SS^{d-1}}vY_{lm}+\frac{1}{r^2}\int_{\SS^{d-1}}v\De_{\SS^{d-1}}Y_{lm}\\
&= \bigg(\pd_{rr} +\frac{d-1}r\pd_r -\frac{l(l+d-2)}{r^2}\bigg)v_{lm}\,.
\end{align*}
Note that we have integrated by parts the term involving the spherical Laplacian.

Thus we infer that $v_{lm}$ satisfies a Bessel-type ODE that only depends on~$l$. Since $v$ is smooth, $v_{lm}$ must be well behaved at $r=0$, which ensures that there are real constants $c_{lm}$ such that
\begin{equation*}
v_{lm}(r)= c_{lm}J_{l+\frac{d-2}{2}}(r)r^{1-d/2}\,.
\end{equation*}
Furthermore, $c_{lm}=0$ for all odd $l$ because $v$ is assumed to be even, so the integral defining $v_{lm}(r)$ is zero.

Since $v$ is smooth, the above series converges to~$v$ e.g.\ in $L^2(B')$. Therefore, for any $\delta>0$ there is an integer $l_0$ such that the finite sum \begin{equation*}w:=\sum_{l=0}^{l_0}\sum_{m=1}^{d_l}c_{lm}J_{l+\frac{d-2}{2}}(r)r^{1-d/2}Y_{lm}(\omega)\end{equation*}
approximates the function $v$ as
\begin{equation*}
\|v-w\|_{L^2(B')}<\delta\,.
\end{equation*}
As the difference $v-w$ also satisfies the Helmholtz equation, standard elliptic estimates show that the approximation also holds
in the $C^k$ sense:
 \begin{equation}\label{1}\|v-w\|_{C^k(B)}<C\delta\,.\end{equation}

To connect these expressions with eigenfunctions of $H_\ka\harm$, let us start with the following lemma:
\begin{lemma}\label{lema1}
Let us fix some integers $l$ and $n$ as in (\ref{rHE}). Uniformly for $r\leq R$, the eigenfunction $\psi^{\ka,\mathrm{H}}_{nlm}$ admits the asymptotic expansion, as $n\rightarrow\infty$ and $\ka\rightarrow0$,
\begin{equation*}
\psi^{\ka,\mathrm{H}}_{nlm}\left(\frac{r}{\sqrt{\lambda_{nl}\harm}},\omega\right)=A_{nl}\left[J_{l+\frac{d-1}{2}}\left( r\right)\cdot r^\frac{2-d}{2}+\mathcal{O}\left(n^{-\frac{d+1}{4}}\right)+\mathcal{O}(\ka^2)\right]Y_{lm}(\omega)\,.
\end{equation*}
In fact,
\begin{equation}\label{app2}
\lim_{n \rightarrow \infty,\ka\rightarrow0}\left\|\psi^{\kappa,\mathrm{H}}_{nlm}\left(\frac{r}{\sqrt{\lambda_{nl}\harm}},\omega\right)-A_{nl}J_{l+\frac{d-1}{2}}\left( r\right)\cdot r^\frac{2-d}{2}Y_{lm}(\omega)\right\|_{C^k(B_R)}=0\,.
\end{equation}
Throughout, we assume $2n+l<\frac{\be}{\ka^2}-\frac{d-1}{2}$. Here $A_{nl}$ is a non-zero constant.
\end{lemma}
\begin{proof}
The result is proved in two steps. In the first part of the proof, we use the asymptotic relation between Laguerre and Jacobi polynomials that can be found, for example, in \cite[(5.3.4.)]{Szego}:\begin{equation*}
	L_{n}^{(\alpha)}(r)=\lim _{\gamma \rightarrow \infty} P_{n}^{(\alpha ; \gamma)}\left(1-2 \gamma^{-1} r\right)\,.
\end{equation*}
It is well-known that one can rewrite Jacobi and Laguerre polynomials as hypergeometric functions and confluent hypergeometric functions respectively:
\begin{equation*}
	\begin{array}{c}
		P_{n}^{(\alpha ; \gamma)}(z)=\left(\begin{array}{c}
			n+\alpha \\
			n
		\end{array}\right){ }_{2} F_{1}\left(-n, 1+\alpha+\gamma+n ; \alpha+1 ; \frac{1}{2}(1-z)\right)\,, \\
		L_{n}^{(\alpha)}(z):=\left(\begin{array}{c}
			n+\alpha \\
			n
		\end{array}\right) M(-n, \alpha+1, z)={ }_{1} F_{1}(-n, \alpha+1, z) \,.
	\end{array}
\end{equation*}
If we pay attention to the fact that the first argument in both of them is a non-positive integer, we can conclude that the series defining hypergeometric functions are just finite sums and then we can bound the error:
\begin{multline*}
	\left|P_{n}^{(\alpha ; \gamma)}\left(1-\frac{2r}{\gamma}\right)-	L_{n}^{(\alpha)}(r)\right|\\
	=\left(\begin{array}{c}
		n+\alpha \\
		n
	\end{array}\right)\left|{ }_{2} F_{1}\left(-n, 1+\alpha+\gamma+n ; \alpha+1 ; \frac{r}{\gamma}\right)- M(-n, \alpha+1, r)\right|\\ =\left(\begin{array}{c}
	n+\alpha \\
	n
\end{array}\right)\left|\sum_{k=0}^{n+1}\frac{(-n)_kr^k}{(\al+1)_kk!}\left(\frac{(1+\al+\gamma+n)_k}{\gamma^k}-1\right)\right|\\=
\left(\begin{array}{c}
n+\alpha \\
n
\end{array}\right)\left|\frac{(-n)r(1+\al+n)}{(\al+1)\gamma}+\frac{(-n)(1-n)r^2((1+\al+n+\ga)(2+\al+n+\ga)-\ga^2)}{2(\al+1)(\al+2)\ga^2}\right.\\
\left.+\cdots+\frac{(-1)r^n((1+\al+n+\ga)\cdots(2n+\al+\ga)-\ga^n)}{(\al+1)\cdots(\al+n)n!\ga^n}\right|=\mathcal{O}\left(\ga^{-1}\right),
\end{multline*}
uniformly for $0\leq r\leq R$ as $\ga\rightarrow\infty$. Here $(q)_{n}$ is the (rising) Pochhammer symbol:
$$
(q)_{n}:=\left\{\begin{array}{ll}
1\,, & n=0 \\
q(q+1) \cdots(q+n-1)\,, & n>0
\end{array}\right.
$$ From this, it follows that \begin{equation*}
	P_{n}^{(\alpha ; \gamma)}\left(1-2 \gamma^{-1} r\right)=L_{n}^{(\alpha)}(r)+\mathcal{O}\left(\gamma^{-1}\right),
\end{equation*} uniformly in $[0,R]$.

If we take $\gamma:=-\beta/\ka^2-1/2$ and notice that $\mathcal{O}(\gamma^{-1})=\mathcal{O}(\ka^2)$, we can conclude \begin{multline*}
		 P_{n}^{\left(l+d/2-1 ;-\beta\kappa^{-2}-1 / 2\right)}\left(1+2 \kappa^{2} r^{2}\right)
		= P_{n}^{(l+d / 2-1 ; \gamma)}\left(1-2 r^{2} \frac{\beta}{(1 / 2+\gamma)}\right) \\
		= P_{n}^{(l+d / 2-1 ; \gamma)}\left(1-2\beta r^{2} \gamma^{-1}+\mathcal{O}\left(\gamma^{-2}\right)\right) = L_{n}^{(l+d / 2-1)}\left(\beta r^{2}\right)+\mathcal{O}\left(\kappa^{2}\right)\\
		= L_{n}^{(l+d / 2-1)}\left(\sqrt{\al} r^{2}+\mathcal{O}(\ka^2)\right)+\mathcal{O}\left(\kappa^{2}\right)= L_{n}^{(l+d / 2-1)}\left(\sqrt{\al} r^{2}\right)+\mathcal{O}\left(\kappa^{2}\right).
	\end{multline*}
	Here we have used that \begin{equation*}\be=\frac{\sqrt{\ka^4+4\alpha}-\ka^2}{2}=\sqrt{\al}+\mathcal{O}(\ka^2)\,.\end{equation*}On the other hand, some elementary computations show that \begin{equation*}\left(\sqrt{1+\kappa^{2} r^{2}}\right)^{-\beta\kappa^{-2}}=e^{-\sqrt{\al}r^{2} / 2+\mathcal{O}(\ka^2)}+\mathcal{O}\left(\kappa^{2}\right)=e^{-\sqrt{\al}r^{2} / 2}+\mathcal{O}(\ka^2)\,,\end{equation*}uniformly on $[0,R]$, since $e^{-r^2\sqrt{\al}/2}>\mathcal{O}(\ka^2)$ for sufficiently small~$\ka$.

Standard formulas for the derivatives of orthogonal polynomials (see e.g. \cite[(4.21.7)]{Szego} and
\cite[(22.8.6)]{Hand}) yield similar approximation results for derivatives of arbitrary order, so we conclude that  \begin{equation*}\left\|f^{\ka,\mathrm{H}}_{nl}(r)-r^l e^{-\sqrt{\alpha}r^2/2}L_n^{\left(l+\frac{d}{2}-1\right)}\left(\sqrt{\alpha}r^2\right)\right\|_{C^k(0,R)}=\mathcal{O}(\ka^2)\,.\end{equation*}
By noticing that \begin{equation*}\lambda^{\ka,\mathrm{H}}_{nl}=\lambda^{\mathrm{H}}_{nl}+\mathcal{O}(\ka^2),\end{equation*}we can see that \begin{equation*}\left\|f^{\ka,\mathrm{H}}_{nl}\left(r/\sqrt{\lambda^{\ka,\mathrm{H}}_{nl}}\right)-\left(\frac{r}{\sqrt{\lambda^{\mathrm{H}}_{nl}}}\right)^l e^{-\sqrt{\alpha}r^2/\left(2\lambda^{\mathrm{H}}_{nl}\right)}L_n^{\left(l+\frac{d}{2}-1\right)}\left(\sqrt{\alpha}\frac{r^2}{\lambda^{\mathrm{H}}_{nl}}\right)\right\|_{C^k(0,R)}=\mathcal{O}(\ka^2)\,.
\end{equation*}

Next we use Hilb's asymptotic formula for the Laguerre polynomial \cite[(8.22.4)]{Szego}, which gives an asymptotic expansion for a fixed $\theta>0$: \begin{equation*}
	e^{-x / 2} x^{\theta / 2} L_{n}^{(\theta)}(x)=\frac{\Gamma(n+\theta+1)}{\left(n+\frac{\theta+1}{2}\right)^{\theta / 2} n !} J_{\theta}\left(\sqrt{(4 n+2 \theta+2) x}\right)+x^{5 / 4} \mathcal{O}\left(n^{\frac{2 \theta-3}{4}}\right)\,.
\end{equation*}
The bound holds uniformly in $0\leq x\leq R$ for any fixed $R>0$. Using the substitutions $x=\sqrt{\al}r^2(\lambda_{nl}^{\text{H}})^{-1}$ and $\theta=l+\frac{d}{2}-1$, we obtain the following asymptotic expansion of the radial part of $\psi\harm_{nlm}$:
\begin{equation*}\Bigg(\frac{r}{\sqrt{\lambda_{nl}\harm}}\Bigg)^le^{-\frac{\sqrt{\al}r^2}{2\lambda_{nl}\harm}}L_n^{(l+d/2-1)}\left(\frac{\sqrt{\al}r^2}{\lambda_{nl}\harm}\right)
=A_{nl}r^{1-d/2}J_{l+d/2-1}(r)+\mathcal{O}\left(n^{l/2+d/4-5/4}\right)\,.\end{equation*}
Here \begin{equation*}A_{nl}:=\left(\frac{\sqrt{\lambda_{nl}\harm}}{2}\right)^{-l}\frac{\Gamma(n+l+d/2)2^{d/2-1}}{n!}\,.\end{equation*}
It is standard that this asymptotic formula can be derived term by term, so one obtains the formula for the derivatives~\eqref{app2} that appears in the lemma and similar ones for higher derivatives.

Using Stirling's asymptotic formula for the factorial, \begin{equation*}
n!=\sqrt{2\pi n}\left(\frac{n}{e}\right)^n\left(1+\mathcal{O}\left(\frac{1}{n}\right)\right)\,,
\end{equation*}
and the identity
\begin{equation}\label{duplication}
	\Gamma(n+l+d/ 2)=\frac{\sqrt{\pi}(2 n+2 l+d-1) !}{2^{2 n+2 l+d-2}(n+l+\frac{d-1}{2}) !}\,,
\end{equation}
we can estimate the constant $A_{nl}$ for large $n$ as\begin{equation*}A_{nl}=\frac{\al^{-l/4}2^{d/2-1}}{e^{l+d/2-1}}n^{l/2+d/2-1}+\mathcal{O}(n^{l/2+d/2-2})\,,\end{equation*}if $d$ is even and \begin{equation*}A_{nl}=\frac{\al^{-l/4}2^{d/2-1}}{e^{l+d/2-1/2}}n^{l/2+d/2-1}+\mathcal{O}(n^{l/2+d/2-2})\,,\end{equation*}if $d$ is odd. Thus we conclude that $n^{l/2+d/4-5/4}A_{nl}^{-1}=\mathcal{O}\left(n^{-\frac{d+1}{4}}\right)$. The lemma then follows by combining the above identities together.
\end{proof}

To continue, let us take a large integer $\widehat{n}$ that will be fixed later, and which we assume to be much larger than $l_0/2$. For each even integer $l$ smaller than $2\widehat{n}$ we set\begin{equation*}\widehat{n}_l:=\widehat{n}-l/2,\end{equation*} so that the eigenvalue \begin{equation}\label{lamb}\lambda:=\lambda_{\widehat{n}_ll}^{\ka,\mathrm{H}}=\be(4\widehat{n}_l+2l+d)-\ka^2\left(2\widehat{n}_l+l+\frac{d-1}{2}\right)^2=
\be(4\widehat{n}+d)-\ka^2\left(2\widehat{n}+\frac{d-1}{2}\right)^2\end{equation} does not depend on the choice of $l$. We can now derive an eigenfunction of the hyperbolic harmonic oscillator from the function $w$ by setting
\begin{equation*}\psi:=\sum_{l=0}^{l_0}\sum_{m=1}^{d_l}c_{lm}A_{\widehat{n}_ll}^{-1}\psi^{\ka,\mathrm{H}}_{\widehat{n}_llm}\,.\end{equation*}
To make sure that this is indeed an eigenfunction, we need to ensure that all the numbers $\widehat n_l,l,m$ are in the admissible range of the integers, so we pick the curvature of the hyperbolic space small enough so that
\begin{equation*}\ka^2<\frac{\beta}{2\widehat{n}+l_0+\frac{d-1}{2}}\,.\end{equation*}
By construction, $\psi$ is then an eigenfunction of the hyperbolic harmonic oscillator with energy as in (\ref{lamb}). Here we have used the fact that $c_{lm}=0$ for odd $l$, since the defined number $\widehat{n}_l$ is an integer only for even~$l$. We are now in conditions to use Lemma \ref{lema1}.

We claim that for any $\delta>0$ one can choose $\widehat{n}$ large enough so that \begin{equation}\label{2}\left\|\psi\left(\cdot/\sqrt{\lambda}\right)-w\right\|_{C^k(B)}<\delta\,.\end{equation} This is a rather straightforward consequence of Lemma \ref{lema1}. Indeed, substituting the asymptotic expressions obtained in the sum for $\psi$ we find \begin{equation*}
	\begin{aligned}
	\left|\psi\left(\frac{x}{\sqrt{\lambda}}\right)-w(x)\right|&\leq \sum_{l=0}^{l_0}\sum_{m=1}^{d_l}|c_{lm}|\left|A_{\widehat{n}_ll}^{-1}\psi^{\ka,\mathrm{H}}_{\widehat{n}_llm}\left(\frac{x}{\sqrt{\lambda}}\right)-J_{l+d/2-1}(r)r^{1-d/2}Y_{lm}(\om)\right|\\
		 &=\sum_{l=0}^{l_0}\sum_{m=1}^{d_l}|c_{lm}|\left(\mathcal{O}\left(\widehat{n}_l^{-\frac{d+1}{4}}\right)+\mathcal{O}(\ka^2)\right)\leq\frac{C}{\widehat{n}^{3/4}}\,,
	\end{aligned}
\end{equation*}
provided that $\widehat{n}$ is much larger than $l_0/2$, $\ka\ll1$ and $|x|<R$, $R$ being the diameter of $B$. An analogous argument shows similar bounds for the derivatives of $\psi$ and $w$ so the estimate~\eqref{2} follows provided $\widehat{n}$ is large enough.

To conclude, we combine~\eqref{1}, \eqref{2} and the fact that we are working in normal geodesic coordinates on a hyperbolic space to show that
\begin{align*}
\left\|v(\cdot)-\psi\circ \exp_{p_0}^\ka\left(\cdot/\sqrt{\lambda}\right)\right\|_{C^k(B)}&\leq\|v-w\|_{C^k(B)}+\left\|\psi\circ \exp_{p_0}^\ka\left(\cdot/\sqrt{\lambda}\right)-w(\cdot)\right\|_{C^k(B)}\\
&<\ep,
\end{align*}provided that $\ka$ is small enough. Here $p_0$ is the point where the harmonic oscillator potential attains its global minimum.

A straightforward modification of the argument enables us to consider the case where $v$ is odd. If $v$ is odd, we only need to notice that $c_{lm}=0$ for all even $l$ and define $\widehat{n}_l$ as $\widehat{n}_l=\widehat{n}-\frac{l+1}{2}$ to prove the result.

\section{The Coulomb operator on hyperbolic spaces}
\label{S.Coul}
The localization result for the Coulomb operator of $\HH^d(\ka)$ analogous to Theorem~\ref{T1} is the following:
\begin{theorem}\label{T.Coul}
	Let $v$ satisfy the equation\begin{equation*}
		\Delta v+\frac{\alpha}{|x|}v=0
	\end{equation*}in $\RR^d$ with $d\geq 2$. Given $\ep>0$ small, an integer $k$ and a ball $B\subset\RR^d$ whose closure does not contain the origin, there exist some $\ka>0$ and an eigenfunction $\psi$ of $H\coul_\ka$ such that	\begin{equation*}
\|\psi\circ \exp_{p_0}^\ka-v\|_{C^k(B)}<\ep\,.
\end{equation*}
Here $p_0\in \HH^d(\ka)$ is the singularity of the Coulomb potential.
\end{theorem}
The proof of this result is very similar to the one of Theorem~\ref{T1}, so we will just sketch it and point out the differences. Arguing as before, we immediately see that~$v$ can be written in the ball $B$ as
\begin{equation*}v(r,\omega)=\sum_{l=0}^\infty\sum_{m=1}^{d_l}c_{lm}J_{2l+d-2}\left(\sqrt{4\alpha r}\right)r^{1-d/2}Y_{lm}(\omega)\,.\end{equation*}
 In the series above, $c_{lm}$ are real constants. We also conclude that for any $\delta>0$ there is an integer $l_0$ such that the truncated sum \begin{equation*}w:=\sum_{l=0}^{l_0}\sum_{m=1}^{d_l}c_{lm}J_{2l+d-2}\left(\sqrt{4\alpha r}\right)r^{1-d/2}Y_{lm}(\omega)\end{equation*}
approximates the function $v$ in $C^k$:  \begin{equation}\label{alpha}\|v-w\|_{C^k(B)}<C\delta\,.\end{equation}

The next lemma is the analog of Lemma \ref{lema1} and will be key in the proof. In the statement we take two constants $R_0<R$ such that $B$ is contained in $B_R\backslash \overline{B_{R_0}}$, where $B_R$ denotes the ball of (geodesic) radius $R$ centered at the origin.

\begin{lemma}\label{lema2}
	Let us fix some integers $l$ and $n$ as in (\ref{rCE}). Uniformly for $0<R_0\leq r\leq R$, the eigenfunction $	\psi^{\kappa,\mathrm{C}}_{nlm}$ admits the asymptotic expansion\begin{equation}\label{cota}
		\begin{aligned}
			\psi^{\kappa,\mathrm{C}}_{nlm}(r,\omega)&=A_{nl}\left[J_{2l+d-2}\left(\sqrt{4\alpha r}\right)r^{1-d/2}+\mathcal{O}\left(n^{-l-d/2+1/4}\right)+\mathcal{O}(\ka)\right]Y_{lm}(\omega)\,,
		\end{aligned}
	\end{equation}as $n\rightarrow\infty$ and $\ka\rightarrow 0$ with $n+l<\sqrt{\frac{\al}{2\ka}}-\frac{d-1}{2}$, where $A_{nl}$ is a nonzero constant. In fact, \begin{equation*}\lim_{n \rightarrow \infty,\ka\rightarrow0}\left\|\psi^{\kappa,\mathrm{C}}_{nlm}(r,\omega)-A_{nl}J_{2l+d-2}\left(\sqrt{4\alpha r}\right)r^{1-d/2}Y_{lm}(\omega)\right\|_{C^k(B)}=0\,.\end{equation*}
\end{lemma}
\begin{proof}
Again, we use the asymptotic relation between Laguerre and Jacobi polynomials \cite[(5.3.4)]{Szego}, from which it follows that \begin{equation*}
		P_{n}^{(\alpha ; \gamma)}\left(1-2 \gamma^{-1} r\right)=L_{n}^{(\alpha)}(r)+\mathcal{O}\left(\gamma^{-1}\right)
	\end{equation*} uniformly in $[0,R]$.
	
	If we take $\gamma=-\frac{2l+2n+d-1}{2}-\frac{\al}{2\ka(n+l+\frac{d-1}{2})}$ and notice that $\mathcal{O}(\gamma^{-1})=\mathcal{O}(\ka)$, we conclude\begin{multline*}
			 P_{n}^{\left(2l+d-2 ;-l-n-\frac{d-1}{2}-\frac{\al}{2\ka(n+l+\frac{d-1}{2})}\right)}\left(2\left(\sqrt{1+\ka^2r^2}+\ka r\right)^2-1\right) \\
			= P_{n}^{(2l+d-2; \gamma)}\left(1+4\ka r+\mathcal{O}(\ka^2)\right) = P_{n}^{(2l+d-2; \gamma)}\left(1-\frac{2r\al}{(n+l+\frac{d-1}{2})\gamma}+\mathcal{O}\left(\gamma^{-2}\right)\right) \\
			=L_{n}^{(2l+d-2)}\left(\frac{\al r}{n+l+\frac{d-1}{2}}\right)+\mathcal{O}\left(\kappa\right).
		\end{multline*}
		On the other hand, using Taylor's series,  \begin{multline*}	\left(\sqrt{1+\kappa^{2} r^{2}}+\ka r\right)^{-n-\frac{\al}{2\ka\left(n+l+\frac{d-1}{2}\right)}}=\left(1+\kappa r+\mathcal{O}\left(\kappa^{2}\right)\right)^{-n-\frac{\al}{2\ka\left(n+l+\frac{d-1}{2}\right)}}\\=\left(1-\frac{\al r}{2(n+l+\frac{d-1}{2})}\cdot\left(\frac{-\al}{2\ka\left(n+l+\frac{d-1}{2}\right)}\right)^{-1}+\mathcal{O}\left(\kappa^{2}\right)\right)^{-n-\frac{\al}{2\ka\left(n+l+\frac{d-1}{2}\right)}}\\=e^{-\frac{\al r}{2\left(n+l+\frac{d-1}{2}\right)}}+ \mathcal{O}\left(\kappa\right)\,,
	\end{multline*}
	also uniformly on $[0,R]$.
	
	Standard formulas for the derivatives of orthogonal polynomials (see e.g. \cite[(4.21.7)]{Szego} and
	\cite[(22.8.6)]{Hand}) allow us to show similar approximation results for derivatives of arbitrary order, and then \begin{equation*}
		\left\|	\psi^{\kappa,\mathrm{C}}_{nlm}-	\psi^{\mathrm{C}}_{nlm}\right\|_{C^k(B)}=\mathcal{O}(\ka)\,.
	\end{equation*} Now using again Hilb's asymptotic formula for the Laguerre polynomial \cite[(8.22.4)]{Szego} and the substitutions $x=\frac{\alpha r}{n+l+(d-1)/2}$ and $\theta=2l+d-2$ we are able to obtain the asymptotic expansion of $\psi\coul_{nlm}$:

\begin{equation*}r^l e^{\frac{-\al r}{2\left(n+l+\frac{d-1}{2}\right)}}L_n^{(2l+d-2)}\left(\frac{\alpha r}{n+l+\frac{d-1}{2}}\right)=A_{nl}J_{2l+d-2}\left(\sqrt{4\alpha r}\right)r^{1-d/2}+\mathcal{O}\left(n^{l+d/2-7/4}\right)\,,\end{equation*}
uniformly in $[R_0,R]$, with \begin{equation*}A_{nl}:=\frac{(n+2l+d-2)!}{k!\,\al^{l-1+d/2}}\,.\end{equation*}

Using Stirling's asymptotic formula for the factorial and the identity~\eqref{duplication}, we can estimate the constant $A_{nl}$ for large $n$ as\begin{equation*}A_{nl}=\frac{n^{2l+d-2}}{e^{2l+d-2}\al^{l+d/2-1}}+\mathcal{O}(n^{2l+d-3})\,.\end{equation*}

 Of course, since the eigenfunctions satisfy the radial equation \begin{equation*}\left(\frac{d^2}{dr^2}+\frac{d-1}{r}\frac{d}{dr}-\frac{l(l+n-2)}{r^2}+\frac{\al}{r}+\lambda_{nl}\coul\right)f_{nl}\coul(r)=0\,,\end{equation*}
 this uniform estimate can be easily promoted to a $C^k$ bound in $[R_0,R]$, and hence in $B$. The bound~\eqref{cota} follows by combining the above identities.
\end{proof}

The main difference with Lemma~\ref{lema1} is that now we are not rescaling eigenfunctions with a factor depending on the increasing energy, that is, we are not approximating in arbitrarily small balls. This can be understood as evidence of the fact that the key ingredient of the proof is only the degeneracy of eigenfunctions, rather than having arbitrarily large energies.

Now we follow the proof of Theorem~\ref{T1}. Let us take a large enough natural number $\widehat{n}$ that will be fixed later, which we assume to be much larger than $l_0$. For each integer $l$ smaller than $\widehat{n}$ we set\begin{equation*}\widehat{n}_l:=\widehat{n}-l\,,\end{equation*} so that the eigenvalue \begin{equation}\label{b}\lambda:=\lambda^{\kappa,\mathrm{C}}_{\widehat{n}_ll}
=-\frac{\al^2}{4(\widehat{n}+\frac{d-1}{2})^2}-\ka^2\left(\widehat{n}+\frac{d-1}{2}\right)^2\end{equation} does not depend on the choice of $l$. Note that, contrary to what happens in the analysis of the harmonic oscillator, no parity hypothesis is needed. Finally, we choose an eigenfunction of the Coulomb operator by setting
\begin{equation*}\psi:=\sum_{l=0}^{l_0}\sum_{m=1}^{d_l}c_{lm}A_{\widehat{n}_ll}^{-1}\psi^{\kappa,\mathrm{C}}_{\widehat{n}_llm},\end{equation*}
with eigenvalue $\lambda$ as in~\eqref{b}. For this function to make sense we need to ensure the existence of eigenfunctions $\psi^{\kappa,\mathrm{C}}_{nlm}$ because of the finite spectrum of the Coulomb operator of $\HH_\ka\coul$. Thus, we impose a restriction over the size of $\ka$: \begin{equation*}\ka<\frac{\al}{2}\left(n+l+\frac{d-1}{2}\right)^{-2}\,.\end{equation*}
Applying Lemma \ref{lema2}, we conclude that for any $\delta>0$ there exists a natural number $\widehat{n}$ large enough so that \begin{equation}\label{beta}\left\|\psi-w\right\|_{C^k(B)}<\delta\,.\end{equation} Indeed, substituting the asymptotic expressions previously obtained in the definition of $\psi$ we find \begin{equation*}
	\begin{aligned}
		\left\|\psi-w\right\|_{C^k(B)}&\leq \sum_{l=0}^{l_0}\sum_{m=1}^{d_l}|c_{lm}|\left\|A_{\widehat{n}_ll}^{-1}\psi^{\kappa,\mathrm{C}}_{\widehat{n}_llm}(r,\omega)-J_{2l+d-2}\left(\sqrt{4\alpha r}\right)r^{1-d/2}Y_{lm}(\omega)\right\|_{C^k(B)}\\
		 &=\sum_{l=0}^{l_0}\sum_{m=1}^{d_l}|c_{lm}|\left(\mathcal{O}\left(\widehat{n}_l^{-l-d/2+1/4}\right)+\mathcal{O}(\ka)\right)<\delta\,,
	\end{aligned}
\end{equation*}
provided that $\widehat{n}$ is large enough and $\ka$ is small.

To conclude, we combine (\ref{alpha}), (\ref{beta}) and the fact that we are working in normal geodesic coordinates on the hyperbolic space centered at the singularity $p_0$ of the Coulomb potential to infer that
\begin{equation*}
	\begin{aligned} &\left\|v-\psi\circ \exp_{p_0}^\ka\right\|_{C^k(B)}\leq\|v-w\|_{C^k(B)}+\left\|w-\psi\circ \exp_{p_0}^\ka\right\|_{C^k(B)}<\ep\,,
	\end{aligned}
\end{equation*}provided that $\ka$ is small enough.

\section{Global approximation with decay}
\label{S.GAT}

In this section we prove a global approximation theorem with decay for local solutions of the Helmholtz equation on manifolds with infinite injectivity radius at some point $p_0$ and that are hyperbolic outside a compact set (Theorem \ref{T2}). As an application of this result, we obtain a global approximation theorem for the heat equation with compactly supported Cauchy data on the hyperbolic space $\HH^d(\ka)$; this extends the analogous result for the Cauchy problem of the heat equation in Euclidean space proved in~\cite{Duke}.

\subsection{Proof of Theorem \ref{T2}}

Let us denote by $N$ the neighbourhood of $\overline\Omega$ where $w$ is defined. We start by taking a smooth function $\chi:M\rightarrow\mathbb{R}$ which is equal to $1$ in a neighbourhood $\Omega'$ of $\overline\Omega$ and identically zero outside $N$. We define the smooth function $w_1$ in $M$ given by $w_1:=\chi w$, which means that $w_1$ equals $0$ outside $N$ even though $w$ is not defined there.

It is well known~\cite{Uhlenbeck} that a generic smooth bounded domain $N'$ of $M$ has the property that $\lambda$ is not a Dirichlet eigenvalue of the Laplacian $\Delta_{M}$ on $N'$. Accordingly~\cite{Browd}, we can find a symmetric Dirichlet Green's function $G$ satisfying the distributional equation:
\begin{equation*}
	\begin{aligned}
		&\Delta_{M,x}G(x,y)+\lambda G(x,y)=\delta_y(x)\ \ \ \text{for all }x,y\in N'\times N',\\
		&G_{\big|\partial N'}(\cdot,y)=0\ \ \forall y\in N'.
	\end{aligned}
\end{equation*}
We also know that $G$ is bounded by a constant times $\rho_M(x,y)^{2-d}$ if $d\geq 3$ and by a constant times $\log\left(\rho_M(x,y)\right)$ if $d=2$, and is smooth outside the diagonal $\text{diag}(N'\times N')$.

In what follows we take $N'$ big enough to contain $\overline{N}$ and we also assume that $N'\backslash \overline\Om$ is connected (e.g., take $N'$ to be a large enough geodesic ball). Let $\dVol$ be the volume measure on $M$. Since $\lambda$ is not a Dirichlet eigenvalue of $\Delta_M$ on $N'$, it readily follows that
\begin{equation}\label{5.5}
	w_1(x)=\int_{N'}G(x,y)f(y)\,\dVol(y)\,,
\end{equation}
for all $x\in N'$, $f$ being the smooth function $f:=\Delta_{M} w_1+\lambda w_1$ supported on $N\backslash\overline{\Omega'}$. A standard continuity argument allows us to approximate the integral~\eqref{5.5} uniformly on $\Omega'$ by a Riemann sum of the form
\begin{equation}\label{5.6}
	w_2(x):=\sum_{n=1}^{n_{\max}}c_nG(x,x_n)\,,
\end{equation}
which is defined for all $x\in N'$. Concretely, for every $\delta>0$ there exist a large enough integer $n_{\max}$, real numbers $c_n$ and points $x_n\in N\backslash \overline{\Omega'}$ such that the finite sum~\eqref{5.6} satisfies the bound
\begin{equation*}
	\parallel w_1-w_2\parallel_{C^0(\Omega')}<\delta\,.
\end{equation*}

Let us now take a big enough geodesic ball $B_R$ centered at $p_0$ so that it contains $N$ (and the compact set $K_1$ on whose complement $M$ is isometric to $\HH^d(\kappa)$) but satisfying $\overline{B_R}\subset N'$ (this is possible because $N'$ is an arbitrary generic domain). We claim that we can sweep the singularities of the function $w_2$ outside $N\backslash\overline{\Omega'}$, in order to approximate it by another function $w_2'$ whose singularities are contained in the complement of $\overline{B_R}$. The proof is based on a duality argument and the Hahn-Banach theorem.

\begin{lemma}\label{vualayage}
	For every $\delta>0$, there is a finite amount of points  $\left\{x_n'\right\}_{n=1}^{n'_{\max}}$ in $N'\backslash \overline{B}_R$ and constants $c_n'$ such that the finite linear combination \begin{equation}\label{5.7}
		w_2'(x):=\sum_{n=1}^{n'_{\max}}c_n'G(x'_n,x)
	\end{equation}approximates the function $w_2$ uniformly in $\Omega$: \begin{equation}\label{5.8}
		\parallel w_2'-w_2\parallel_{C^k(\Omega)}<\delta\,,
	\end{equation}
for any fixed integer $k$.
\end{lemma}
\begin{proof}
	Consider the space $V$ of all finite linear combinations of the form (\ref{5.7}) where $x_n'$ can be any point in  $N'\backslash\overline{B}_R$ and the constants $c_n'$ take arbitrary values. Restricting these functions to the set $\Omega'$, we can regard $V$ as a subspace of the Banach space $L^2(\Omega'):=\left\{f\in L^2(M): \supp f\subseteq \overline{\Omega'}\right\}$.	
	It is well know that the space $L^2(\Omega')$ is its own dual. Let us take any function $g\in L^2(\Omega')$ such that $\int_{\Omega'}f g=0$ for all $f\in V$, that is, $g$ is orthogonal to the subspace $V$. We define a function $F\in L^2(N')$ as
\begin{equation*}
		F(x):=\int_{N'}G(x,y)g(y)\,\dVol(y)=\int_{\Omega'}G(x,y)g(y)\,\dVol(y),
	\end{equation*}so that $F$ satisfies the equation \begin{equation*}
		\Delta_{M}F+\lambda F=g
	\end{equation*}
on $N'$.
	
	Notice that $F$ is identically zero on $N'\backslash\overline{B}_R$ by the definition of the function $g$ and that $F$ satisfies the elliptic equation $\Delta_{M}F+\lambda F=0$ in $N'\backslash \overline{\Omega'}$. Since $N'\backslash\overline{\Omega'}$ is connected provided that $\Omega'$ is close enough to $\Omega$, and contains the set $N'\backslash \overline{B_R}$, by elliptic analytic continuation we conclude that the function $F$ must vanish on the whole $N'\backslash\overline{\Omega'}$. It then follows that, for $y\notin \overline{\Omega'}$,
\begin{equation*}
		0=F(y)=\int_{N'}G(y,x)g(x)\,\dVol(x)\,.
	\end{equation*}
Therefore,
\begin{equation*}
		\int_{N'}w_2(y)g(y)\,\dVol(y)=0\,,
	\end{equation*}
which implies that $w_2$ cannot be separated from the space $V$ and then can it be uniformly approximated on $\Omega'$ by elements of the subspace $V$, due to the Hahn-Banach theorem. Acordingly, there is a finite set of points $\left\{x_n'\right\}_{n=1}^{n'_{\max}}$ in $N'\backslash\overline{B}_R$ and reals $c_n'$ such that the function (\ref{5.7}) satisfies the estimate
	\begin{equation}\label{L2}
		\left\|w_2'-w_2\right\|_{L^2(\Omega')}<\delta\,.
	\end{equation} Standard elliptic estimates on manifolds allow us to promote~\eqref{L2} to a $C^k$ bound by restricting the domain to $\overline{\Omega}\subset\Omega'$, see~\eqref{5.8}, and the lemma follows.
\end{proof}

Next, we notice that the function $w'_2$ satisfies
\begin{equation*}
\Delta_Mw'_2+\lambda w'_2=0\
\end{equation*}
on the ball $B_R$, whose interior contains $\Omega'$ and $\Om$. Exploiting the fact that the injectivity radius of $M$ at $p_0$ is infinite, let us consider spherical geodesic coordinates on $B_R$; that is $\rho:=\rho_M(x,p_0)$ and $\omega\in\mathbb{S}^{d-1}$. Expanding $w'_2$ with respect to the angular variables in a series of spherical harmonics, essentially as in the proof of Theorem~\ref{T1}, we can write
\begin{equation}\label{eq.series}
w'_2=\sum_{l=0}^\infty\sum_{m=1}^{d_l}w_{lm}(\rho)Y_{lm}(\omega)
\end{equation}
with
\begin{equation*}
w_{lm}(\rho):= \int_{\SS^{d-1}} w(\rho,\om)\, Y_{lm}(\om)\, d\si(\om)\,,
\end{equation*}
where $d\si$ is the canonical measure on~$\SS^{d-1}$. The series converges in $L^2(B_R)$.

Using the assumption that $M\backslash K_1$ is isometric to $\HH^d(\kappa)\backslash K_2$, we infer that we can identify a one-sided neighborhood $\{R-\eta<\rho<R\}$ of $\partial B_R$ in $M$ with the corresponding neighborhood in $\HH^d(\kappa)$, for some small enough $\eta>0$. A simple computation analogous to that of the proof of Theorem~\ref{T1} shows that the functions $w_{lm}$ are solutions to the radial ODE
\begin{equation}\label{eq:radial}	
\pd_{\rho\rho}w_{lm}+(d-1)\ka\frac{\cosh(\ka\rho)}{\sinh(\ka \rho)}\pd_\rho w_{lm}-\ka^2\frac{l(l+d-2)}{\sinh(\ka \rho)^2}w_{lm}\lambda w_{lm}=0\,,
\end{equation}
on the interval $(R-\eta,R)$. Therefore, $w_{lm}$ must be a linear combination of the form
\begin{multline*}
		w_{lm}(\rho)=\frac{C_{lm}}{\sinh(\ka \rho)^{d/2-1}}P^{1-d/2-l}_{\nu}\left(\cosh(\ka\rho)\right)\\+\frac{C_{lm}'}{\sinh(\ka\rho)^{d/2-1}}\left(e^{-i\pi\left(\frac{d}{2}-1+l\right)}Q^{d/2-1+l}_\nu\left(\cosh(\ka\rho)\right)\right),
\end{multline*}where  $\nu=-\frac{1}{2}-\frac{i}{2\kappa}\sqrt{4\lambda-(d-1)^2\kappa^2}$ (recall that $\lambda>\left[\frac{(d-1)\ka}{2}\right]^2$), $C_{lm},C_{lm}'$ are complex constants,  and $P_\nu^\mu$ and $Q_\nu^\mu$ are associated Legendre conical functions of the first and second kind, respectively.

A simple observation allows us to infer that the series~\eqref{eq.series} actually converges in $H^2(B_R)$ (or any other Sobolev space). Indeed, as $M$ is hyperbolic outside a compact set, and denoting by $\Delta_\ka$ the Laplacian on the corresponding hyperbolic space, we know that
\begin{equation*}
\Delta_\ka w'_2+\la w'_2=\Theta
\end{equation*}
for some smooth function $\Theta$ whose support is contained in $B_{R-\eta}$. Therefore, expanding $\Theta$ in a series of spherical harmonics using the spherical geodesic coordinates $(\rho,\om)$, we easily deduce that the partial sums
\begin{equation*}
\sum_{l=0}^{l_0}\sum_{m=1}^{d_l}(\Delta_\ka+\la)[w_{lm}(\rho)Y_{lm}(\omega)]
\end{equation*}
converge in $L^2(B_R)$ to $\Theta$ as $l_0\to\infty$. As $B_R$ is bounded, the choice of the elliptic operator and of the volume form used to define Sobolev spaces is irrelevant, so the claim follows.

The $H^2(B_R)$ convergence of the series~\eqref{eq.series} implies that for any $\delta>0$ there is an integer $l_0$ such that the finite sum
\begin{equation}\label{eq.v'}
	v':=\sum_{l=0}^{l_0}\sum_{m=1}^{d_l}w_{lm}(\rho)Y_{lm}(\theta)
\end{equation}
is close to $w'_2$ in the sense:
\begin{equation}\label{eq:Hk}
\parallel v'-w'_2\parallel_{H^2(B_R)}<\delta\,.
\end{equation}
In fact, since the solutions to the radial equation~\eqref{eq:radial} are well defined (and smooth) for all $\rho>R-\eta$, this allows us to extend the function $v'$ beyond $B_R$ to define a function (that we still denote by $v'$) on the whole $M$. It is clear that
\begin{equation*}
\Delta_M v'+\lambda v'=0
\end{equation*}
on $M\backslash \overline{B_{R-\eta}}$. Moreover, by the asymptotic properties of the conical Legendre functions (see, e.g. \cite[Chapter 14]{Hand}), $v'$ has the pointwise decay
\begin{equation}\label{estimrad}
	|v'(x)|< Ce^{\frac{1-d}{2}\ka \rho}\,.
\end{equation}

However, the function $v'$ does not satisfy, in general, the Helmholtz equation on the whole $M$. To address this issue, we set
\begin{equation*}
f_1:=\Delta_M v'+\lambda v'\,,
\end{equation*}
which is a function that is supported in the interior of $B_{R-\eta}$, and is bounded as
\begin{equation*}
\|f_1\|_{L^2(M)}=\|f_1\|_{L^2(B_{R-\eta})}<C\delta\,,
\end{equation*}
as a consequence of the estimate~\eqref{eq:Hk}. Following~\cite[Section 6]{Perry}, we can define the function
\begin{equation*}
f_2:=(\Delta_M+\lambda+i0)^{-1} f_1
\end{equation*}
for $\lambda>\left[\frac{(d-1)\ka}{2}\right]^2$, which satisfies the equation $(\Delta_M+\lambda)f_2=f_1$ on $M$ and the Agmon--H\"ormander type bound
\begin{equation}\label{estimAH}
\sup _{R>0} \frac{1}{R} \int_{B_{R}}|f_2(x)|^{2} \,\dVol(x)<C\delta \,.
\end{equation}
See~\cite[Section 6]{Perry} for the precise meaning of the resolvent operator $(\Delta_M+\lambda+i0)^{-1}$ and the corresponding estimates.

Finally, if we define the function
\begin{equation*}
v:=v'-f_2
\end{equation*}
we deduce that it satisfies the Helmholtz equation $\Delta_M v+\lambda v=0$ on $M$, and decays at infinity as
\begin{equation*}
\sup _{R>0} \frac{1}{R} \int_{B_{R}}|v(x)|^{2} \,\dVol(x)<\infty
\end{equation*}
by the estimates~\eqref{estimAH} and~\eqref{estimrad} (it is straightforward to check, using the expression of the volume form on the hyperbolic space in geodesic coordinates, that the pointwise decay~\eqref{estimrad} satisfies an Agmon--H\"ormander type bound). Moreover, all the previous estimates allow us to write

\begin{equation*}
\|v-w\|_{L^2(\Om')}\leq \|f_2\|_{L^2(\Om')}+\|v'-w'_2\|_{L^2(\Om')}+\|w'_2-w_2\|_{L^2(\Om')} +\|w_2-w\|_{L^2(\Om')}<C\delta\,.
\end{equation*}

Finally, since $v$ and $w$ satisfy the Helmholtz equation in $\Omega'$, standard elliptic estimates imply that this $L^2$ bound can be promoted to a $C^k$ bound \begin{equation*}
	\parallel w-v\parallel_{C^k(\Omega)}<\ep\,.
\end{equation*}
The theorem follows observing that when $M$ is the hyperbolic space $\HH^d(\kappa)$, the truncated series $v'$, cf. Equation~\eqref{eq.v'}, does satisfy the Helmholtz equation on the whole space (because each summand satisfies it), and hence the function $f_2$ is identically zero on $M$. Accordingly, $v=v'$ and satisfies the pointwise bound~\eqref{estimrad}, as claimed.

\subsection{The heat equation}

In this subsection, for simplicity we will restrict ourselves to the case when the manifold is the hyperbolic space $M=\HH^d(\ka)$. To state our approximation theorem for the heat equation, given a function $v_0\in C^\infty_0(\HH^d(\ka))$, let us denote by $v:=e^{t\Delta_\ka}v_0$ the only solution to the Cauchy problem
\begin{equation}\label{heat}
	\pd_t w-\De_\ka w=0
\end{equation} on $\HH^d(\ka)\times (0,\infty)$ and
\begin{equation*}
	w(x,0)=v_0(x)\,,
\end{equation*}
for all $x\in \HH^d(\ka)$,  such that $w(x,t)$ tends to zero as $\rho_\ka(x)\rightarrow \infty$ for all $t>0$. It is well known that $w$ can be written in terms of the hyperbolic heat kernel $K(t,x,y)$ as
\begin{equation*}
	w(x,t)=\int_{\HH^d(\ka)}C_dK(t,x,y)v_0(y)\,\dVol(y)\,,
\end{equation*}where $C_d$ is a dimensional constant.

In the Appendix we show that this kernel depends only on the geodesic distance $\rho\equiv\rho_\ka(x,y)$ between $x$ and $y$, and we find an explicit formula for it. We hence use the notation $K(t,x,y)=H(t,\rho)$, with $H(t,\rho)=0$ if $t<0$ and for $t>0$ it is given by the following formula.

If the dimension $d$ is odd, $d=2m+1$, then
\begin{equation*}
	H(t,\rho)=\frac{(-1)^m}{2^m\pi^m\sqrt{4\pi t}}\left(\frac{\ka}{\sinh(\ka \rho)}\frac{\partial}{\partial \rho}\right)^m e^{-\ka^2m^2t-\rho^2/4t}\,.
\end{equation*}
If the dimension $d$ is even, $d=2 m+2$, then
\begin{equation*}
	H(t,\rho)=\frac{(-1)^me^{-\frac{(2m+1)^2\ka^2t}{4}}\ka}{t^{3/2}2^{m+5/2}\pi^{m+3/2}}\left(\frac{\ka}{\sinh(\ka \rho)}\frac{\partial}{\partial \rho}\right)^m\int_{\rho}^\infty\frac{se^{-s^2/4t}ds}{\sqrt{\cosh(\ka s)-\cosh(\ka \rho)}}\,.
\end{equation*}
We also state some pointwise bounds of this kernel and its derivatives in the Appendix.
Now, given a space-time domain $\Om\subset \HH^d(\ka)\times \RR$, we denote by
\begin{equation*}
\Om[t]:=\{x\in \HH^d(\ka): (x,t)\in\Om \}
\end{equation*}
its intersection with the time~$t$ slice. The following global approximation theorem for the heat equation on the hyperbolic space extends an analogous result proved in~\cite{Duke} in the Euclidean case.

\begin{theorem}\label{T.heat}
	Let $w$ satisfy Equation~\eqref{heat} in a neighbourhood of the closure of a bounded domain $\Om \subset \HH^d(\ka)\times (0,\infty)$. Suppose that the complement of $\overline{\Om[t]}$ in $\HH^d(\ka)$ is connected for all~$t$ and fix an integer~$k$ and $\ep>0$. Then there exists $v_0 \in C^\infty_0(\HH^d(\ka))$, such that the function $v(x,t):= e^{t\De_\ka}v_0(x)$ approximates~$w$ as
	\begin{equation*}
	\|v-w\|_{C^k(\Om)}<\ep\,.
	\end{equation*}
\end{theorem}

\begin{proof}
We denote by $N$ a neighborhood of $\Om$ where $w$ is defined. By the hypoellipticity of the parabolic equations we know that $w\in C^\infty(N)$. Taking geodesic spherical coordinates centered at some point of $\HH^d(\ka)$, we can assume that $N$ is contained in $B_R\times (0,T)$, for some geodesic ball $B_R$ of large enough radius $R$. Taking a smooth cut-off function $\chi:\HH^d(\ka)\times\RR\rightarrow\RR$ that is equal to $1$ in a neighborhood $\Om'\subset N$ of $\overline \Om$ and is identically zero outside $N$, we infer that the function $w_1:=\chi w$ is compactly supported and satisfies the equation
\begin{equation*}
		\partial_t w_1-\Delta_{\ka}w_1=\phi\,,
\end{equation*}
for some smooth function $\phi$ on $\HH^d(\ka)\times \RR$ which is supported on $N\backslash\overline{\Om'}$. Using the properties of the heat kernel it is easy to check that
 \begin{equation*}
		w_1(x,t)=\int_{\HH^d(\ka)\times \RR}C_dH(t-s,\rho_\ka(x,y))\phi(s,y)\,ds\,\ dVol(y)\,.
\end{equation*}

If $K$ is any compact set with nonempty interior that is contained in the complement of $\overline{B_R}\times \RR$ and whose time projection\begin{equation*}\left\{t\in\RR: (x,t)\in K\text{ for some }x\in \HH^d(\ka)\right\}\end{equation*}
contains that of $N$, we claim that there is a finite sum
\begin{equation*}
		w_2(x,t):=\sum_{j=0}^{j_ {\max}}c_jH(t-s_j,\rho_\ka(x,y_j)),
	\end{equation*}
with $c_j$ real constants and $(y_j,s_j)\in K$ such that
\begin{equation}\label{eq.w1w2}
		\left\|w_1-w_2\right\|_{C^k(\Om)}<\delta,
	\end{equation}
for any fixed $k$ and $\delta>0$. To prove this statement we use the next two lemmas.
\begin{lemma}\label{Lm1}
		Let $U$ be a domain in $\HH^d(\ka)\times \RR$ such that $U[t]$ is connected for all $t\in \RR$. Consider a point $(y,s)\in U$ and a bounded domain $K\subset U$ such that $K[s]\neq\emptyset$. Then, for any $\epsilon>0$ and any bounded domain $W\subset U^c$ there exist a finite set of points $\{(y_j,s_j)\}_{j=1}^J$ in $K$ and real constant $\{b_j\}_{j=1}^J\subset\RR$ such that
\begin{equation*}
			\left\|	H(\cdot-s,\rho_\ka(\cdot,y))-\sum_{j=1}^Jb_jH(\cdot-s_j,\rho_\ka(\cdot,y_j))\right\|_{C^k(W)}<\epsilon\,.
		\end{equation*}
\end{lemma}
\begin{proof}
We assume that $(y,s)$ does not belong to $K$, as otherwise the statement is trivial. Let us take a proper bounded subdomain $U_1\subset U$ containing $(y,s)$ and $K$. We can assume that $U_1[t]$ is connected for all $t$. Consider the space $\mathcal S$ of all finite linear combinations of the fundamental solution with poles belonging to $K$, that is, \begin{equation*}
		\mathcal S:=\spn_\RR\left\{H(\cdot-\tau,\rho_\ka(\cdot,z)): (z,\tau)\in K\right\}\,.
	\end{equation*}
	Restricting these functions to a neighborhood $W'$ of $W$ contained in the complement of $\overline{U_1}$, $\mathcal S$ can be regarded as a subspace of the Banach space $L^2(W')$, which is its own dual. The lemma follows arguing as in the proof of Lemma~\ref{vualayage}, to conclude that there exists a function \begin{equation*}
		v:=\sum_{j=1}^Jb_jH(\cdot-s_j,\rho_\ka(\cdot,y_j))
	\end{equation*}
in $\mathcal S$ such that
\begin{equation*}
		\left\|H(\cdot-s,\rho_\ka(\cdot,y))-v\right\|_{C^k(W)}< \epsilon\,.
	\end{equation*}
For this, one just needs to use the anisotropic unique continuation theorem for parabolic equations (see e.g.~\cite[Theorem 2.1]{Duke}) and apply standard parabolic estimates to promote an $L^2(W')$ bound to a $C^k(W)$ bound.
\end{proof}

The second lemma follows from an easy continuity argument that makes use of the properties of the heat kernel $H(t,\rho)$ (see the Appendix), so we omit its proof.

\begin{lemma}\label{Lm2}
Let $\varphi:\HH^d(\ka)\times\RR\rightarrow\RR$ be a smooth function of compact support. For any bounded domain $W$ contained in the complement of $\supp \varphi$ and any $\epsilon>0$, there exist a finite set of points $\{(y_j,s_j)\}_{j=1}^J$ in $\supp\varphi$ and constants $\{c_j\}_{j=1}^J$ such that \begin{equation*}
			 \left\|\int_{\HH^d(\ka)\times\RR}H(\cdot-s,\rho_\ka(y,\cdot))\varphi(y,s)\,\dVol(y)\,ds-\sum_{j=1}^Jc_jH(\cdot-s_j,\rho_\ka(y_j,\cdot))\right\|_{C^k(W)}<\epsilon\,.
		\end{equation*}
\end{lemma}

Now, by Lemma~\ref{Lm2}, taking $U=(\HH^d(\ka)\times \RR)\backslash \overline\Om$ and $W=\Om$, we can approximate the function $w_1$ by a linear combination of fundamental solutions with poles contained in the set $N\backslash \overline{\Om'}$. Hence, noticing that $U[t]$ is connected for all $t$, we can use Lemma~\ref{Lm1} to construct a function $w_2$ satisfying the estimate~\eqref{eq.w1w2} with poles contained in a compact set $K$ of the complement of $\overline{B_R}\times \RR$, as claimed.

Next, notice that the decay properties of the fundamental solution $H(t,\rho)$ (see the Appendix) imply that for large $|t|$
\begin{equation*}
		\sup_{x\in B_R}\left|D^\alpha_x\partial^k_tw_2(x,t)\right|\leq C_\alpha |t|^{-3/2} e^{\frac{-|t|\ka^2(d-1)^2}{4}} \,.
	\end{equation*}
Then $w_2$ satisfies the uniform $L^1$ bound
\begin{equation*}
		\sup_{x\in B_R}\int_{-\infty}^{\infty}\left|D^\alpha_x\partial^k_tw_2(x,t)\right|dt<C_{\alpha,k}\,,
	\end{equation*}
for any $\alpha$ and $k$. For all $x\in B_R$, the mapping properties of the Fourier transform ensure that the Fourier transform of $w_2$ with respect to time,\begin{equation*}
		\widehat{w}_2(x,\tau):=\frac{1}{2\pi}\int_{-\infty}^\infty w_2(x,t)e^{-i\tau t}dt\,,
	\end{equation*}
is bounded and depends continuously on $\tau$. Since this holds for all derivatives of $w_2$ with respect to $t$, we infer that its Fourier transform falls off as
\begin{equation*}
		\sup_{x\in B_R}\left|D^\alpha_x\widehat{w}_2(x,\tau)\right|<\frac{C_n}{1+|\tau|^n}
	\end{equation*}
for any $n$. Of course, $\widehat{w}_2$ is a smooth function of $x\in B_R$ because so is $w_2$. This implies that the inverse Fourier transform formula \begin{equation}\label{conv}
		D^\alpha_x \partial_t^kw_2(x,t)=\int_{-\infty}^\infty (i\tau)^kD_x^\alpha \widehat{w}_2(x,\tau)e^{i\tau t}d\tau
	\end{equation}
holds pointwise for $(x,t)\in B_R\times\RR$. In particular, as $\partial_tw_2-\Delta_{\ka}w_2=0$ in that set, it follows that
\begin{equation}\label{eq.tau}
\Delta_{\ka}\widehat{w}_2(x,\tau)-i\tau\widehat{w}_2(x,\tau)=0
\end{equation}
for all $(x,\tau)\in B_R\times \RR$.
	
	We next expand $\widehat{w}_2(x,\tau)$, with $x\in B_R$ for each $\tau\in\RR$, in a basis of spherical harmonics on the unit sphere $\mathbb S^{d-1}$, which we denote as usual by
\begin{equation*}
		\left\{Y_{lm}(\omega):l\geq 0, 1\leq m\leq d_l\right\}
	\end{equation*}
with $\omega\in \mathbb{S}^{d-1}$, and assume to be normalized so that they are an orthonormal basis of $L^2(\SS^{d-1})$. We recall that $\Delta_{\SS^{d-1}}Y_{lm}=-\mu_lY_{lm}$, where $\mu_l=l(l+d-2)$ is an eigenvalue of the Laplacian on $\mathbb S^{d-1}$ of multiplicity \begin{equation*}d_l=\frac{2l+d-2}{l+d-2}{l+d-2\choose l}\,.\end{equation*}

More precisely, $\widehat{w}_2$ takes the form:
	\begin{equation}\label{conv2}
		\widehat{w}_{2}(x, \tau)=: \sum_{l=0}^{\infty} \sum_{m=1}^{d_{l}} \varphi_{l m}(\rho, \tau) Y_{l m}(\omega)\,,
	\end{equation}
where now $\rho\equiv\rho_\ka$ is the geodesic distance to the center of the ball $B_R$ and $\omega=x/\rho$. Notice that the coefficient $\varphi_{lm}(\rho,\tau)$ is given by
\begin{equation*}
		\varphi_{l m}(\rho, \tau)=\int_{\mathbb{S}^{d-1}} \widehat{w}_{2}(\rho, \omega, \tau) Y_{l m}(\omega) d \sigma(\omega),
	\end{equation*}where $d\sigma$ is the standard measure on the unit sphere, so $\varphi_{l m}(\rho,\tau)$ is a $C^\infty$ function of $\rho\in(0,R)$ for all $\tau$. This series converges in $L^2(B_R)$.

In particular, writing Equation~\eqref{eq.tau} in spherical geodesic coordinates, it follows that
\begin{equation*}
		\begin{aligned}
			&0 =\Delta_M \widehat{w}_{2}(x, \tau)-i \tau \widehat{w}_{2}(x, \tau) \\
			&=\sum_{l=0}^{\infty} \sum_{m=1}^{d_{l}}\left[\partial_{\rho\rho} \varphi_{lm}+(d-1)\ka\frac{\cosh(\ka \rho)}{\sinh(\ka \rho)} \partial_{\rho} \varphi_{lm}-\left(i \tau+\frac{\mu_{l}\ka^2}{\sinh(\ka \rho)^{2}}\right) \varphi_{lm}\right] Y_{lm}(\omega)
		\end{aligned}
	\end{equation*}
for all $x\in B_R$ (in the $L^2$ sense). Therefore $\varphi_{lm}$ satisfies the radial ODE
\begin{equation*}
		\partial_{\rho\rho} \varphi_{lm}+(d-1)\ka\frac{\cosh(\ka \rho)}{\sinh(\ka \rho)} \partial_{\rho} \varphi_{lm}-\left(i \tau+\frac{\mu_{l}\ka^2}{\sinh(\ka \rho)^{2}}\right) \varphi_{lm}=0
	\end{equation*}
for $0<\rho<R$ and stays bounded at $\rho=0$. The only solution to this ODE bounded at $\rho=0$ is
\begin{equation*}
		\varphi_{l m}(\rho,\tau)=A_{lm}(\tau)\left(\frac{\sinh(\ka\rho)}{\ka}\right)^{1-d/2}Q^{1-l-d/2}_\mu\left(\cosh(\ka\rho)\right)\,,
	\end{equation*}
where $A_{lm}(\tau)$ is a complex constant that may depend on $\tau$ and
\begin{equation*}\mu:=-\frac{1}{2}-\frac{1}{2\ka}\sqrt{\ka^2(d-1)^2-4i\tau}\,.\end{equation*}

We notice that the solution is an associated Legendre function of the second kind which satisfies a exponential bound:
\begin{equation}\label{cotas}
	\left|\varphi_{lm}(\rho,\tau)\right|\leq C(\tau)e^{A(\tau)\rho},
	\end{equation}
for all $\rho\in(0,\infty)$ with $A(\tau)=\frac{(1-d)\ka+\left(\ka^4(d-1)^4+16\tau^2\right)^{1/4}}{2}$.

An additional important property is that each summand of the series~\eqref{conv2} is a smooth function on $\HH^d(\ka)$ for each $\tau\in\RR$, which satisfies the Equation~\eqref{eq.tau} on the whole hyperbolic space. Standard elliptic estimates then imply that the sum~\eqref{conv2} converges on $C^k(B)$ for any integer $k$ and any smaller ball $B\subset B_R$, and the convergence is uniform for $(\rho,\tau)$ in compact subsets of $[0,R)\times\RR$. For future reference, we will fix some ball $B$ such that $\Omega\subset B\times (0,T)$.

	In view of the good convergence properties of the integral~\eqref{conv} and of the series~\eqref{conv2}, it is not hard to see that for any $k$ and any $\delta>0$ one can choose large enough $L$ and $\tau_0$ such that
\begin{equation*}
		\left\|w_2-w_3\right\|_{C^k(B\times(-T,T))}<\delta\,,
	\end{equation*}
where
\begin{equation*}
		w_3(x,t):=\sum_{l=0}^L\sum_{m=1}^{d_l}\int_{-\tau_0}^{\tau_0}\varphi_{lm}(\rho,\tau)Y_{lm}(\omega)e^{i\tau t}d\tau\,.
	\end{equation*}
By the previous properties, we infer that
	\begin{equation*}
		\frac{\partial w_3}{\partial t}-\Delta_{\ka}w_3=0\,,
	\end{equation*}
in $\HH^d(\ka)\times \RR$ and $w_3$ is bounded like
\begin{equation*}
	\sup _{t \in \mathbb{R}}\left|w_{3}(x, t)\right|<C e^{A\rho}
\end{equation*}
as a consequence of (\ref{cotas}).
	
Now let us consider the smooth function\begin{equation*}
		f(x):=w_3(x,0)\,.
	\end{equation*}
As $\partial_tw_3-\Delta_\ka w_3=0$ and $w_3(x,0)=f(x)$, the fact that $f$ and $w_3$ satisfy the exponential bound
$$|f(x)|+|w_3(x,t)|<Ce^{A\rho}$$
permits to invoke Grigor'yan's uniqueness theorem for the heat equation on manifolds~\cite[Lemma 2.1]{Mura} to conclude that
\begin{equation*}
	w_3(x,t)=\int_{\HH^d(\ka)}C_dH(t,\rho_M(x,y))f(y)\,\dVol(y)
	\end{equation*}
for $t>0$ and $x\in\HH^d(\ka)$. As the integral converges uniformly in $C^k$, one can take a smooth compactly supported function $v_0(\cdot):=\chi_1(\epsilon\cdot)f(\cdot)$, where $0\leq\chi_1\leq 1,\chi_1=1$ on $B_1$ and $\chi_1=0$ outside $B_2$, to conclude that
\begin{equation*}
		\left\|w_3-v\right\|_{C^k(B\times(0,T))}<\delta\,,
	\end{equation*}
	for any small enough $\epsilon$, where
\begin{equation*}
		v(x,t):=\int_{\HH^d(\ka)}C_dH(t,\rho_\ka(x,y))v_0(y)\,\dVol(y)\,.
	\end{equation*}
Putting together all the previous bounds, we finally obtain
\begin{equation*}
		\left\|v-w\right\|_{C^k(\Omega)}<C\delta\,,
	\end{equation*}
and the theorem then follows.
\end{proof}

\section*{Acknowledgements}

This work has received funding from the European Research Council (ERC) under the European Union's Horizon 2020 research and innovation programme through the Consolidator Grant agreement~862342 (A.E.\ and A.G.-R.). It is partially supported by the grants CEX2019-000904-S, RED2018-102650-T and PID2019-106715GB GB-C21 (D.P.-S.) funded by MCIN/AEI/10.13039/501100011033.

\section*{Appendix A. The heat kernel of a hyperbolic space}

In this Appendix we recall explicit formulas for the heat kernel of $\HH^d(\ka)$ (see, for example, \cite{Mandou,GrigoryanI}).
If the dimension $d$ is odd, $d=2m+1$, then
\begin{equation*}
	H_d^\ka(t,\rho)=\frac{(-1)^m}{2^m\pi^m\sqrt{4\pi t}}\left(\frac{\ka}{\sinh(\ka \rho)}\frac{\partial}{\partial \rho}\right)^m e^{-\ka^2m^2t-\rho^2/4t}\,.
\end{equation*}
If the dimension $d$ is even, $d=2 m+2$, then
\begin{equation*}
	H_d^\ka(t,\rho)=\frac{(-1)^me^{-\frac{(2m+1)^2\ka^2t}{4}}\ka}{t^{3/2}2^{m+5/2}\pi^{m+3/2}}\left(\frac{\ka}{\sinh(\ka \rho)}\frac{\partial}{\partial \rho}\right)^m\int_{\rho}^\infty\frac{se^{-s^2/4t}ds}{\sqrt{\cosh(\ka s)-\cosh(\ka \rho)}}\,.\end{equation*}

Two important recurrence formulas for heat kernels in hyperbolic spaces of different dimensions are  \begin{equation*}
	H_{d+2}^\ka(t,\rho)=-\frac{e^{-d\ka^2 t}\ka}{2 \pi \sinh (\ka \rho)} \frac{\partial}{\partial \rho} H_{d}^\ka(t,\rho)
\end{equation*}
also \begin{equation*}
	H_d^\ka(t,\rho)=\int_{\rho}^\infty \frac{e^{\frac{(2d-1)t\ka^2}{4}}H_{d+1}^\ka(t,\mu)\sinh(\ka\mu)\sqrt{2}d\mu}{\sqrt{\cosh(\ka\mu)-\cosh(\ka\rho)}}\,.
\end{equation*}
From this integral formula, one infers that there exists a positive constant $c=c(d,\ka)$ such that\begin{equation*}0\leq H_d^\ka(t,\rho)\leq  B(t,\rho)\,,\end{equation*}for all $t,\rho>0$, where \begin{equation*}B(t,\rho):= c e^{-\frac{\ka^2(d-1)^2t}{4}-\frac{(d-1)\ka\rho}{2}-\frac{\rho^2}{4t}}\frac{(1+\ka\rho+\ka^2t)^{\frac{(d-3)}{2}}(1+\ka\rho)}{(4\pi t)^{d/2}}\,.\end{equation*}In particular, for each fixed $\rho>0$, the heat kernel and its derivatives decay for large positive time as
\begin{equation*}|\partial^n_\rho\partial^k_tH_d^\ka(t,\rho)|\leq C_n t^{-3/2}e^{-\frac{t\ka^2(d-1)^2}{4}}\,.\end{equation*}

\bibliographystyle{APA}

\end{document}